\documentclass[a4paper,12pt]{amsart}

\usepackage{float}
\usepackage{euscript,eufrak,verbatim}
\usepackage{graphicx}
\usepackage{amscd,pifont}
\usepackage[usenames]{color}
\usepackage[colorlinks,linkcolor=red,anchorcolor=blue,citecolor=blue]{hyperref}
\usepackage{amsmath}
\usepackage{amsthm}
\usepackage[all]{xy}
\usepackage{graphicx} 

\newtheorem{thm}{Theorem}[section]

\newtheorem{lem}[thm]{Lemma}
\newtheorem{cor}[thm]{Corollary}
\newtheorem{que}[thm]{Question}

\newtheorem{conj}[thm]{Conjecture}

\newtheorem{red}[thm]{Reduction}

\def\N{\mathbb{N}}

\def\dimB{{\rm dim~}}

\def\dimLow{\underline{\rm dim}}
\def\dimE{{\rm dim}_{\text{e}}}
\def\Z{\mathbb{Z}}
\def\N{\mathbb{N}}
\def\R{\mathbb{R}}

\makeatletter 
\@addtoreset{equation}{section}
\makeatother  

\newcommand{\norm}[1]{\left\lVert#1\right\rVert}

\title[On dimensions of frame spectral measures and their frame spectra]{On dimensions of frame spectral measures and their frame spectra}
\author{Ruxi Shi
}
\address{Institute of Mathematics, Polish Academy of Sciences, ul. \'Sniadeckich 8, 00-656 Warszawa, Poland}
\email{rshi@impan.pl}


\begin{document}
	
	\maketitle


\begin{abstract}
In this paper, we prove that the entropy dimension of a frame spectral measure is superior than or equal to the Beurling dimension of its frame spectrum.
\end{abstract}

\section{Introduction}

A set $\Lambda$ in a Hilbert space $\mathcal{H}$ is called a \textit{frame} if there exist two constants $A,B>0$ such that for every $f\in \mathcal{H}$, we have
\begin{equation}\label{eq:frame}
A\norm{f}^2\le \sum_{\lambda\in \Lambda}|\langle f, \lambda \rangle|^2\le B \norm{f}^2,
\end{equation}
where $\langle \cdot, \cdot \rangle$ is the inner product in $\mathcal{H}$. The constants $A$ and $B$ are called \textit{lower} and \textit{upper bounds} of the frame. Moreover, if only the upper bound hold in (\ref{eq:frame}), then we call $\Lambda$ a \textit{Bessel set} or \textit{Bessel sequence} in $\mathcal{H}$. It is not hard to see that frame is a natural generalization of orthonormal basis  (where $A = B = 1$). If we restrict $\mathcal{H}=L^2(\mu)$ for some Borel measure $\mu$ on a locally compact abelian group $G$, then $\mu$ is called a \textit{frame spectral measure} with \textit{frame spectrum} $\Lambda$ if (\ref{eq:frame}) holds and $\Lambda$ is contained in the dual group $ \widehat{G}$, and furthermore called \textit{a spectral measure} if $A=B=1$.

The notion of frame was introduced by Duffin and Schaeffer \cite{DufSch1952} in the context of nonharmonic Fourier series. Frames provide robust, basis-like (but non-unique) representations of vectors in a
Hilbert space. The potential redundancy of frames often allows one to construct them more easily than bases, and to get better properties than those that are achievable using bases. Nowadays, frames have various applications in a wide range of areas. However, few properties of frame spectral measures are known. In this paper, we are interested in the relation between the dimensions of the frame spectral measure and its spectrum. It is believed that the ``dimension" of the frame spectral measure should control the ``dimension" of its spectrum. But only the case when the measure is self-similar was established (\cite{DutHanSunWeb2011,HeKanTanWu2018}).  In such case, the frame spectral measure is exact dimensional and its Hausdorff dimension controls the beurling dimension (see the definition in Section \ref{Sec:dimension of measure}) of its frame spectrum. In general, it was conjectured in \cite{HeKanTanWu2018} that

\begin{conj}\label{conj 1}
	If $\mu$ is a frame spectral measure with spectrum $\Lambda$ and compact support $T$ then $\dimB \Lambda\le \dim_H T$.
\end{conj}

In this paper, we disprove Conjecture \ref{conj 1}, which means that the Hausdorff dimension is not a candidate that controls the Berling dimension in general. Instead, we prove a similar version of Conjecture \ref{conj 1} by replacing Hausdorff dimension by upper entropy dimension (see its definition in Section \ref{Sec:Beurling dimension of countable sets}). More precisely, supposing that $\mu$ is a frame spectral measure with frame spectrum $\Lambda$, we show that the Beurling dimension  of the frame spectrum $\Lambda$ is not superior than the upper entropy dimension of the frame spectral measure $\mu$, which allows one to see that the Beurling dimension  and the upper entropy dimension  are the proper notions of ``dimensions" which are described above.  Now we state our main result.

\begin{thm}\label{main thm 1}
	Let $\mu$ be a Borel measure on $\R^d$. Suppose that $\mu$ is a frame spectral measure with frame spectrum $\Lambda$. Then we have
	$$
	\dimB \Lambda\le \overline{\dimE}~ \mu,
	$$
	where $\dimB \Lambda$ is the Beurling dimension of $\Lambda$ and $\overline{\dimE}~ \mu$ is the upper entropy dimension of $\mu$.
\end{thm}

The following theorem is formally stronger than  but actually equivalent to Theorem \ref{main thm 1}.

\begin{thm}\label{main thm 2}
	Let $\mu$ be a Borel measure on $\R^d$. Suppose that $\mu$ is a frame spectral measure with frame spectrum $\Lambda$. Then we have
	$$
	\dimB \Lambda\le \inf_{\mu(K)>0, \mu(\partial K)=0}\overline{\dimE}~ \mu_K.
	$$
	where $\mu_K$ is the measure $\mu$ restricted on the Borel set $K$.
\end{thm}

In fact, Theorem \ref{main thm 1} is equivalent to Theorem \ref{main thm 2} as follows: the necessity is trivial, and the sufficiency follows from Lemma \ref{lem:restriction} which states that if $\mu$ is a frame spectral measure then $\mu_K$ is also a frame spectral measure for every Borel set $K$ satisfying $\mu(\partial K)=0$. 

In general, we can not expect that the equality holds in Theorem \ref{main thm 1} or Theorem \ref{main thm 2}. We refer to the examples in \cite{DaiLaiHe2013} where a class of singular continuous measures are constructed, satisfying that the Beurling dimensions of their spectra are zero but their entropy dimensions are strictly positive.  

Even though Theorem \ref{main thm 2} is stated for frame spectral measures, we remark that only the upper bound in (\ref{eq:frame}) plays a role in the proof of Theorem \ref{main thm 2}. In other words, if a Borel measure $\mu$ has a Bessel sequence $\Lambda$, then Theorems \ref{main thm 1} and \ref{main thm 2} hold for $\mu$ and $\Lambda$.

We organize our paper as follows. In Section \ref{Sec:Preliminaries}, we recall several definitions of different dimensions, including Hausdorff dimension, entropy dimension, etc. In order to prove Theorem \ref{main thm 1} in Section \ref{Sec:Proof of main result}, we make some reduction of Theorem \ref{main thm 1} in Section \ref{Sec:Reduction of the main result}. In Section \ref{sec:Further discussion}, we discuss the relation between Beurling dimension with other dimensions.
Finally, in Section \ref{Sec:Potential examples}, we show some application of our main result.



\section{Preliminaries}\label{Sec:Preliminaries}

In this section, we recall several definitions of different dimensions.

\subsection{Dyadic partitions}
We first define the $n$-th dyadic partition of $\R$ by
$$
\mathcal{D}_n^{(1)}:=\left\{ \left[ \frac{k}{2^n}, \frac{k+1}{2^n} \right): k\in \Z  \right\}.
$$
The $n$-th dyadic partition of $\R^d$ is then defined by
$$
\mathcal{D}_n^{(d)}:=\left\{ I_1\times I_2\times \dots \times I_d: I_j\in \mathcal{D}_n^{(1)}  \right\}.
$$
If there is no confusion, we usually omit the superscript and write $\mathcal{D}_n$ for the $n$-th dyadic partition of $\R^d$.

\subsection{Dimensions of measure}\label{Sec:dimension of measure}
Let $(X, \mathcal{B}, \mu)$ be a probability space. Let $\mathcal{A}$ be a partition of $X$. The \textit{Shannon entropy of $\mu$ with respect to $\mathcal{A}$} is defined by
$$
H(\mu,\mathcal{A})=\sum_{A\in \mathcal{A}}-\mu(A)\log \mu(A).
$$
By convention the logarithm is taken in base $2$ and $0\log 0=0$. If the partition $\mathcal{A}$ is infinite, then the entropy $H(\mu, \mathcal{A})$ may be infinite.

Recall that $\mathcal{D}_n$ is the dyadic partition of $\R^d$ with diameter $2^{-n}$. The \textit{entropy dimension} of $\mu$ is defined by the formula
$$
\dimE~ \mu=\lim\limits_{n\to \infty} \frac{1}{n} H(\mu, \mathcal{D}_n),
$$
if the limit exists (otherwise we take limsup or liminf as appropriate, denoted by $\overline{\dim}_e~ \mu$ and $\underline{\dim}_e~ \mu$ respectively).


The \textit{lower Hausdorff dimension} of $\mu$ is defined by
$$
\dimLow_H~ \mu=\inf\{\dim_H A: \mu(A)>0 \},
$$
and the \textit{upper Hausdorff dimension} of $\mu$ is defined by
$$
\overline{\dim}_H~ \mu=\inf\{\dim_H A: \mu(A)=1 \}.
$$
Here $\dim_H A$ is the Hausdorff dimension of $A$. In what follows, the open ball of radius
$r$ centered at $x$ is denoted by $B(x,r)$.
A measure $\mu$ is \textit{exact dimensional} if the local dimension
$$
\lim\limits_{r\to 0} \frac{\log \mu(B(x,r))}{\log r}
$$
exists and is $\mu$-a.e. constant, which is denoted by $\dim \mu$. 
Here and in what follows, we denote by  $\mathcal{P}(\R^d)$ the space of probability measures on $\R^d$. The following lemma is well-known and its proof can be found, for example, in \cite[Theorem 1.1]{FanLauRao2002}.
\begin{lem}
	If $\mu\in \mathcal{P}(\R^d)$  is exact dimensional, then $\dimE ~\mu$ exists and is equal to $\dim \mu$.
\end{lem}


\subsection{Beurling dimension of countable sets}\label{Sec:Beurling dimension of countable sets}
Let $\Lambda$ be a countable set in $\R^d$. For $r>0$, the \textit{upper Beurling density corresponding to $r$} (or \textit{$r$-Beurling density}) of $\Lambda$ is defined by the formula
$$
\mathfrak{D}_r^+:=\limsup_{h\to \infty} \sup_{x\in \R^d} \frac{\sharp(\lambda \cap B(x,h))}{h^r}.
$$
The \textit{(upper) Beurling dimension} of $\Lambda$ is defined by
$$
\dim\Lambda=\sup\{r>0:\mathfrak{D}_r^+(\Lambda)>0 \},
$$
or alternatively,
$$
\dim\Lambda=\inf\{r>0:\mathfrak{D}_r^+(\Lambda)<+\infty \}.
$$
A basic property of Beurling dimension is that $\dim s\Lambda=\dim \Lambda$ for all $s\in \R\setminus\{0\}$.


\section{Reduction of the main result}\label{Sec:Reduction of the main result}

In this section, our goal is to make some reduction of Theorem \ref{main thm 1}.

The following lemma is a direct consequence of the definition of frame spectral measures. We omit the proof and leave the readers to work out the details.
\begin{lem}\label{lem:basic properties}
	Let $\mu$ be a frame spectral measure in $\R^d$ with spectrum $\Lambda$ and frame bounds $0<A\le B<\infty$. Then we have the following properties.
	\begin{itemize}
		\item [(1)] For any $v,t\in \R^d$, $\mu(\cdot+v)$ is a frame spectral measure  with spectrum $\Lambda+t$ and frame bounds $A,B$.
		\item [(2)] For any non-zero $c\in \R$, $c\cdot \mu$ is a frame spectral measure with spectrum $\Lambda$ and frame bounds $cA,cB$.
		\item [(3)] For any non-zero $s\in \R$, $\mu(\cdot \times s)$ is a frame spectral measure with spectrum $s\Lambda$ and frame bounds $A,B$.
	\end{itemize}
\end{lem}

We observe that if putting $\mathcal{H}=L^2(\mu)$ and $f=\overline{\lambda}^{-1}$ for some $\lambda\in \Lambda$ in \eqref{eq:frame}, then we have $\mu(\R^d)<\infty$. Since $c\cdot \mu$ is also a frame spectral measure for any $c\in \R$, we might assume that $\mu\in \mathcal{P}(\R^d)$.

It is well known that (\cite{LabaWang2006,HeLaiLau2013}) that if  $\mu$ is a frame spectral measure with frame $\Lambda$, then it has to be of ``pure type".

\begin{thm}[\cite{HeLaiLau2013}, Theorem 1.1, Proposition 2.1]
	Let $\mu$ be a frame spectral measure with frame $\Lambda$. Then
	$\mu$ is either discrete with $\sharp \Lambda<+\infty$, absolutely continuous with $\mathfrak{D}_d^{-}(\Lambda)>0$ or singular continuous with $\mathfrak{D}_d^{-}(\Lambda)=0$.
\end{thm}

  If $\mu$ is discrete, then it has finitely many atoms \cite{HeLaiLau2013}, implying that $\dim \Lambda=\dim_e \mu=0$. If $\mu$ is absolutely continuous, then it is supported on a set of finite Lebesgue measure in $\R^d$, and its density function is bounded from above and from below almost everywhere on the support \cite{Lai2011}. It follows that if $\mu$ is absolutely continuous, then $\dim \Lambda=\dim_e \mu=d$. By the above argument, it is sufficient to prove Theorem \ref{main thm 1} for singular continuous measures.

Let $\mu\in \mathcal{P}(\R^d)$. For a Borel set $K\subset \R^d$, we denote by 
$$
\mu_{K}(\cdot):=\mu(\cdot\cap K),
$$
the measure $\mu$ restricted on $K$. Moreover, if $K$ is a dyadic cube in $[0,1]^d$ with $\mu(K)>0$, we denote by
$$
\mu_{K}^{\Box}(\cdot):=\frac{1}{\mu(K)}(S_K)_*\mu_{K}(\cdot),
$$
where $S_K$ is the affine bijective map from $K$ to $[0,1]^d$ and $(S_K)_*\mu_{K}(\cdot)$ is the pushforward of $\mu_K$, i.e. the measure $\mu_{K}(S_K^{-1}(\cdot))$. Obviously, we have $\mu_{K}^{\Box}\in \mathcal{P}([0,1]^d)$.

The following lemma provides that the restriction of a frame spectral measure is also a frame spectral measure. A general version of the following lemma can be found in \cite{FuLai2018}. We include the proof here for completeness.
\begin{lem}\label{lem:restriction}
	Let $\mu$ be a frame spectral measure on $\R^d$. Let $K\subset \R^d$ be a Borel subset satisfying that $\mu(\partial K)=0$. Then $\mu_{K}$ is also a frame spectral measure having the same spectrum and frame bounds with the measure $\mu$.
\end{lem}
\begin{proof}
	For any $f\in L^2(\mu_K)$, we extend $f$ into the space $L^2(\mu)$ by taking $f(x)=0$ for any $x$ outside $\overline{K}$. Since $\mu_{K^c}(\overline{K})=0$, we have $\langle f, g \rangle_{\mu_{K^c}}=0$ for all $g\in L^2(\mu)$. It follows that $\langle f, g \rangle_{\mu}=\langle f, g \rangle_{\mu_K}$ for all $g\in L^2(\mu)$, and in particular  $\norm{f}_{\mu}=\norm{f}_{\mu_{K}}$. Thus we conclude that $\mu_{K}$ is also a frame spectral measure and has the same spectrum and frame bounds with the measure $\mu$.
\end{proof}

The following lemma shows that we can find a unit cube satisfying the condition (of $K$) in Lemma \ref{lem:restriction}.

\begin{lem}\label{lem:good point}
	Let $\mu\in \mathcal{P}(\R^d)$. Then there exists a unit cube  $v+[0,1]^d$ for some $v\in \R^d$ such that $\mu(v+[0,1]^d)>0$ and $\mu(\partial( v+[0,1]^d))=0$.
\end{lem}
\begin{proof}
	Since $\mu\in \mathcal{P}(\R^d)$, there exists $u\in \R^d$ such that $\mu(u+[0,1/2]^d)>0$. Let $u_t=u-(t,t,\dots, t)$. Since $u_t+[0,1]^d$ contains $u+[0,1/2]^d$ for any $t\in [0, d^{1/2}/2]$, we see that $\mu(u_t+[0,1]^d)>0$. Observe that $\partial[0,1]^d=A\cup B$ where $A=\{x=(x_i)_{1\le i\le d}: x_i\in [0,1], \exists x_j=0\}$ and $B=\{x=(x_i)_{1\le i\le d}: x_i\in [0,1], \exists x_j=1\}$. Moreover, it is easy to check that $(u_t+A)\cap (u_s+A)=\emptyset$ and $(u_t+B)\cap (u_s+B)=\emptyset$ for distinct $s,t\in [0, d^{1/2}/2]$. By the fact that the sum of an uncountable number of positive numbers is infinite\footnote{\label{foot}Suppose that $I$ is an uncountable set and $a_i>0$ for $i\in I$. Let $C_n=\{a_i: a_i>1/n, i\in I \}$. Then there exists $n\in \N$ such that $C_n$ contains infinite elements. Thus $\sum_{i\in I}a_i\ge \sum_{a_i\in C_n} a_i\ge \sum_{a_i\in C_n} 1/n=+\infty$.}, we get that $\mu(u_t+A)>0$ (resp. $\mu(u_t+B)>0$) for at most countably many $t\in  [0, d^{1/2}/2]$. Thus there exists $t\in [0, d^{1/2}/2]$ such that $\mu(u_t+A)=\mu(u_t+B)=0$ and consequently $\mu(\partial(u_t+[0,1]^d))=0$.
\end{proof}

Let $v$ be as in Lemma \ref{lem:good point}. Then by Lemma \ref{lem:basic properties} and Lemma \ref{lem:restriction}, the measure $$\frac{1}{\mu(v+[0,1]^d)}\mu_{v+[0,1]^d}(\cdot+v)\in \mathcal{P}([0,1]^d)$$ has the same spectrum with $\mu$. Since 
$$
\overline{\dim}_e~\frac{1}{\mu(v+[0,1]^d)}\mu_{v+[0,1]^d}(\cdot+v)\le \overline{\dim}_e~\mu,
$$
it is sufficient to prove Theorem \ref{main thm 1} for measures in $\mathcal{P}([0,1]^d)$. 

The following lemma shows that the dyadic partition is ``nice"  up to a scaling of $\mu$.

\begin{lem}\label{lem: good s}
	Let $\mu\in \mathcal{P}([0,1]^d)$. Then there exists $1\le s< \infty$ such that $\mu(\cdot \times s)\in \mathcal{P}([0,1]^d)$ satisfies that  $\mu((\cdot \times s) \cap \partial D)=0$ for all $D\in \mathcal{D}_n$ where $\{\mathcal{D}_n\}_{n=1}^\infty$ is the set of the dyadic partitions.
\end{lem}
\begin{proof}
	Fix $n\in \N$ and $D\in \mathcal{D}_n$. Suppose $D=u+[0,2^{-n}]^d$ where $u=(u_i)_{1\le i\le d}\in [0,1]^d$. We decompose $\partial D=A\cup B$ where $A=\{x=(x_i)_{1\le i\le d}: x_i\in [u_i,u_i+2^{-n}], \exists x_j=u_j\}$ and $B=\{x=(x_i)_{1\le i\le d}: x_i\in [u_i,u_i+2^{-n}], \exists x_j=u_j+2^{-n}\}$. it is easy to check that $tA\cap sA=\emptyset$ and $tB\cap sB=\emptyset$ for distinct $s,t\in [1, \infty)$. By the fact that the sum of an uncountable number of positive numbers is infinite (See Footnote \ref{foot}), we get that $\mu(tA)>0$ (resp. $\mu(tB)>0$) for at most countably many $t\in  [1, \infty)$. It follows that $\mu(t\partial D)>0$ for at most countably many $t\in  [1, \infty)$. Since $\{\mathcal{D}_n\}_{n=1}^\infty$ consists of countably many elements, there  exists $1\le s< \infty$ such that $\mu(\cdot \times s)$ satisfies that  $\mu((\cdot \times s) \cap \partial D)=0$ for all $D\in \mathcal{D}_n$.
\end{proof}

  By Lemma \ref{lem: good s}, Lemma \ref{lem:basic properties} (3) and the facts that $\dim s\Lambda=\dim \Lambda$ and $\overline{\dim}_e \mu(\cdot \times s)=\overline{\dim}_e \mu$, we might assume $s=1$ for the sake of simplicity.

Finally, we summarize the reductions made so far in the following list.

\begin{red}\label{Red:reduction}
	In order to prove Theorem \ref{main thm 1}, we might assume that a frame spectral measure $\mu$ in $\R^d$ has the following structure:
	\begin{itemize}
		\item [(1)] The measure $\mu$ is singular continuous. 

		\item [(2)] The measure $\mu$ belongs to $\mathcal{P}([0,1]^d)$.

		\item [(3)] The dyadic partitions $\{\mathcal{D}_n\}_{n=1}^\infty$ satisfy that $\mu(\partial D)=0$ for all $D\in \mathcal{D}_n$.
	\end{itemize}
\end{red}

\section{Proof of main result}\label{Sec:Proof of main result}

Let $\mu\in \mathcal{P}([0,1]^d)$. It is well-known that the dual group $\widehat{\R^d}$ consists of exponential functions which is isomorphic to $\R^d$. We could thus identify $\R^d$ with the subspace in $L^2(\mu)$ and write $\left\langle f, \lambda \right\rangle_{\mu}$ for $f\in L^2(\mu)$ and $\lambda\in \R^d$. More precisely, we write the inner products
$$
\left\langle f, \lambda \right\rangle_{\mu}=\int_{[0,1]^d} f(x)e^{-2\pi i \lambda \cdot x} d\mu(x)
$$
and
$$
\left\langle t, \lambda \right\rangle_{\mu}=\int_{[0,1]^d} e^{2\pi i (t-\lambda) \cdot x} d\mu(x),
$$
for $f\in L^2(\mu)$ and $t,\lambda\in \R^d.$

The following two lemmas not only has its own interest but also are useful to prove our main result.
\begin{lem}\label{lem:change measure}
	Suppose that the measure $\mu\in \mathcal{P}([0,1]^d)$ is a frame spectral measure with spectrum $\Lambda$ and frame bounds $0<A\le B<\infty$. Let $n>0$. Then for any $D\in \mathcal{D}_n$ with $\mu(D)>0$, and for any $t\in \R^d$, we have
	\begin{equation*}
	\frac{A}{\mu(D)}\le \sum_{\lambda\in \Lambda} \left|\left\langle \frac{1}{2^n}t, \frac{1}{2^n}\lambda \right\rangle_{\mu_{D}^{\Box}}  \right|^2\le \frac{B}{\mu(D)}.
	\end{equation*}
\end{lem}
\begin{proof}
	Let  $D\in \mathcal{D}_n$ with $\mu(D)>0$. Let $t\in \R^d$. By Lemma \ref{lem:restriction} and Reduction \ref{Red:reduction} (3), we have
	\begin{equation}\label{eq:change measure 1}
	A\mu(D)\le \sum_{\lambda\in \Lambda} \left|\left\langle t, \lambda \right\rangle_{\mu_{D}}  \right|^2\le B\mu(D).
	\end{equation}
	For any $\lambda\in \Lambda$, we observe that
	\begin{equation*}
	\left\langle t, \lambda \right\rangle_{\mu_{D}}=\mu(D) e^{2\pi i(t-\lambda)\cdot v(D)} \left\langle  \frac{1}{2^n}t, \frac{1}{2^n}\lambda \right\rangle_{\mu_{D}^\Box},
	\end{equation*}
	where $v(D)\in \R^d$ is the vector satisfying $D=v(D)+[0, \frac{1}{2^n}]^d$. It follows that
	\begin{equation}\label{eq:change measure 2}
	\left|\left\langle t, \lambda \right\rangle_{\mu_{D}} \right|=\mu(D)\left|  \left\langle  \frac{1}{2^n}t, \frac{1}{2^n}\lambda \right\rangle_{\mu_{D}^\Box} \right|,
	\end{equation}
	for any $\lambda\in \Lambda$.
	Combing \eqref{eq:change measure 1} and \eqref{eq:change measure 2}, we complete the proof.
\end{proof}

\begin{lem}\label{lem:fourier transform}
	Let $\mu\in \mathcal{P}(\R^d)$. Then for any $0<\epsilon<1$, there exists $\delta=\delta(\epsilon)>0$ such that for any $|\xi|<\delta$ and for any $D\in \mathcal{D}_n$ with $\mu(D)>0$, we have
	$$
	\left| \widehat{\mu_{D}^{\Box}}(\xi)  \right|>\epsilon.
	$$
\end{lem}
\begin{proof}
	Let $0<\epsilon<1$. Pick arbitrary $D\in  \mathcal{D}_n$.
	A simple computation shows that 
	\begin{equation}\label{eq:fourier transform 1}
	\left|\widehat{\mu_{D}^{\Box}}(\xi) \right|=\left| \int_{[0,1]^d} e^{2\pi i \xi \cdot x}d\mu_{D}^{\Box}(x)  \right|\ge\left| \int_{[0,1]^d}\cos(2\pi \xi \cdot x)d\mu_{D}^{\Box}(x)  \right|.
	\end{equation}
	We choose $0<\delta<\frac{1}{4d}$ small enough such that $\cos(2d\pi \delta)>\epsilon$. Since $\cos(\theta)$ is positive and decreasing for $\theta\in (0,\frac{\pi}{2})$, we have $\cos(2d\pi \theta)>\epsilon$ for all $0<\theta< \delta$. Then for any $|\xi|<\delta$ and $x\in [0,1]^d$, we have that $|\xi\cdot x|<d\delta$ and consequently that $\cos(2\pi \xi \cdot x)>\epsilon$. It follows from \eqref{eq:fourier transform 1} that
	$$
	\left|\widehat{\mu_{D}^{\Box}}(\xi) \right|>\epsilon.
	$$
	This completes the proof.
\end{proof}

Now we prove our main result.
\begin{proof}[Proof of Theorem \ref{main thm 1}]
	 Pick arbitrary $s>0$. Then there exists $N>0$ such that for any $n>N$, we have
	 \begin{equation*}
	 \overline{\dim}_e \mu +s \ge \frac{1}{n}\sum_{D\in \mathcal{D}_n} -\mu(D)\log \mu(D).
	 \end{equation*}
	 It follows that
	 \begin{equation}\label{eq:main thm 1}
	 (2^{n})^{\overline{\dim}_e \mu +s}\ge 2^{\sum_{D\in \mathcal{D}_n} -\mu(D)\log \mu(D)}=\prod_{D\in \mathcal{D}_n, \mu(D)>0} \mu(D)^{-\mu(D)}.
	 \end{equation}
	 Let $h>2^N$. Then there exists a positive integer $n_h$ such that $2^{n_h-1}< h\le 2^{n_h}$. Let $\epsilon>0$. Let $\delta=\delta(\epsilon)$ which is defined in Lemma \ref{lem:fourier transform}. Let $\rho$ be the minimal integer such that $2^{-\rho}<\delta$. For any $t\in \R^d$ and any $D\in \mathcal{D}_{n_h+\rho}$, we have
	 \begin{equation}\label{eq:main thm 2}
	 \begin{split}
	 \epsilon^2 \cdot \sharp (\Lambda\cap B(t, h))
	 &\le \sum_{\lambda\in \Lambda\cap B(t, h)} \left|\left\langle \frac{1}{2^{n_h+\rho}}t, \frac{1}{2^{n_h+\rho}}\lambda \right\rangle_{\mu_{D}^{\Box}}  \right|^2\\
	 &\le \sum_{\lambda\in \Lambda} \left|\left\langle \frac{1}{2^{n_h+\rho}}t, \frac{1}{2^{n_h+\rho}}\lambda \right\rangle_{\mu_{D}^{\Box}}  \right|^2\le \frac{B}{\mu({D})}.
	 \end{split}
	 \end{equation}
	 Since $\sum_{D\in \mathcal{D}_{n_h+\rho}}\mu(D)=1$ and \eqref{eq:main thm 2} holds for all $D\in \mathcal{D}_{n_h+\rho}$ with $\mu(D)>0$, we have
	 \begin{equation}\label{eq:main thm 3}
	 \begin{split}
	 \sharp (\Lambda\cap B(t, h))
	 &=\prod_{D\in \mathcal{D}_{n_h+\rho}, \mu(D)>0}(\sharp (\Lambda\cap B(t, h)))^{\mu(D)} \\
	 &\le  B\epsilon^{-2} \prod_{D\in \mathcal{D}_{n_h+\rho}, \mu(D)>0} \mu(D)^{-\mu(D)}
	 \end{split}
	 \end{equation}
	 It follows from \eqref{eq:main thm 1} and \eqref{eq:main thm 3} that
	 \begin{equation*}
	 \frac{\sharp (\Lambda\cap B(t, h))}{h^{\overline{\dim}_e \mu +s}}\le \frac{\sharp (\Lambda\cap B(t, h))}{2^{(n_h-1)(\overline{\dim}_e \mu +s)}}\le B\epsilon^{-2}\cdot 2^{(1+\rho)(\overline{\dim}_e \mu +s)}.
	 \end{equation*}
	 Then we deduce that $\dim \Lambda\le \overline{\dim}_e \mu +s$. Since $s$ can be chosen arbitrarily close to $0$, we conclude that $\dim \Lambda\le \overline{\dim}_e \mu$.
\end{proof}

\section{Further discussion}\label{sec:Further discussion}
In this section, we will discuss several different notions of dimensions (or conditions) and its relations with Beurling dimension and entropy dimension.
\subsection{Lev's condition}
Let $\mu\in \mathcal{P}(\R^d)$. Given a real number $\alpha$ with $0\le \alpha\le d$.
Lev considered the following condition in \cite{Lev2018}:
\begin{equation}\label{eq:lev}
\liminf_{r\to \infty} \frac{1}{r^{d-\alpha}}\int_{-r}^r |\widehat{\mu}(t)|^2 dt>0.
\end{equation}
He proved that if  a frame spectral measure $\mu$  with frame spectrum $\Lambda$ satisfies \eqref{eq:lev}, then 
\begin{equation}\label{eq:lev 2}
\sup_{x\in \R^d} \sharp(\Lambda\cap B(x,r))\le Cr^\alpha,
\end{equation}
for some constant $C$ which does not depend on $r$.
We define 
$$
L(\mu)=\inf\{\alpha: \eqref{eq:lev}~\text{holds for}~\mu~\text{and}~\alpha \}.
$$
Therefore we restate \eqref{eq:lev 2} as follows.
\begin{lem}
	Let $\mu$ be a frame spectral measure with frame spectrum $\Lambda$. Then we have
	$$
	\dimB \Lambda\le L(\mu).
	$$
\end{lem}
The value $L(\mu)$ is sometimes related to the ``dimension" of $\mu$. For example, if $\mu$ is a certain self-similar measure with Hausdorff dimension $\alpha$, then $L(\mu)=\alpha$. However, such relation is very difficult to establish and compute in some cases. For instant, as far as I know, the condition \eqref{eq:lev} is unknown for self-affine measures. On the other hand, we will see that the entropy dimension is well established for self-affine measures in the next section.

\subsection{Hausdorff dimension}
For a Borel measure $\mu$, we know the facts that $\underline{\dim}_H \mu\le \overline{\dim}_e\mu$ and that $\overline{\dim}_H\mu$ is not comparable with $ \overline{\dim}_e\mu$ (see for example \cite{FanLauRao2002}). A natural question arises as to whether the entropy dimension can be replaced by the Hausdorff dimension in Theorems \ref{main thm 1} and \ref{main thm 2}. We will give a negative answer to this question in the following.

Let $I\subset \N$ and $p$ be a prime number.
For $n\in \N$, let 
$$
I_n=\{i\in I: i\le n \}
$$
be the finite subset of $I$ and let 
$$
C(I_n)=\left\{\sum _{i\in I_n}b_i p^{-i} :   b_{i}\in \{ 0, 1, \cdots, p-1\} \right\}
$$
be the finite subset of the unit interval $[0,1]$.
It is not hard to see that the weak limit of $\frac{1}{\sharp C(I_n)}\delta_{C(I_n)}$ exists, which is denoted by $\nu_I$, as $n$ tends to infinity. Let 
$$
\Lambda_{I_n}=\left\{\sum _{i\in I_n}b_i p^{i} :   b_{i}\in \{ 0, 1, \cdots, p-1\} \right\}
$$ 
be the finite subset of $\Z$. Obviously, we have the inclusion $\Lambda_{I_1}\subset \Lambda_{I_2}\subset \cdots$. Let 
$$
\Lambda_{I} = \cup_{i\in \N} \Lambda_{I_n}.
$$
In \cite{Shi2019}, the author showed that the measure $\nu_I$ is a spectral measure with spectrum $\Lambda_{I}$. Moreover, it is computed that
$$
\dimB \Lambda_{I}=\overline{\dim}_e ~\nu_I =\limsup_{n\to \infty} \frac{\sharp I_n}{n},
$$
and 
$$
\dim_H \text{supp}(\mu)= \overline{\dim}_H~\nu_I=\underline{\dim}_H ~\nu_I=\liminf_{n\to \infty} \frac{\sharp I_n}{n}.
$$
At the same time, it is shown that
$$
\overline{\dim}_e ~\nu_I=\inf_{\mu(K)>0, \mu(\partial K)=0 } \{\overline{\dim}_e ~{(\nu_I)}_{K}\},
$$
and
\begin{align*}
\dim_H \text{supp}(\mu)
&=\inf_{\mu(K)>0, \mu(\partial K)=0} \{\overline{\dim}_H ~{(\nu_I)}_{K} \}\\
&=\inf_{\mu(K)>0, \mu(\partial K)=0} \{\underline{\dim}_H ~{(\nu_I)}_{K} \}.
\end{align*}
It is not hard to pick suitable $I\subset \N$ such that 
\begin{equation}\label{eq:strict larger}
\liminf_{n\to \infty} \frac{\sharp I_n}{n}<\limsup_{n\to \infty} \frac{\sharp I_n}{n}.
\end{equation}
Under the condition \eqref{eq:strict larger}, we have 
$$
\dimB \Lambda_{I}=\overline{\dim}_e ~\nu_I>\dim_H \text{supp}(\mu)= \overline{\dim}_H~\nu_I=\underline{\dim}_H ~\nu_I.
$$
This disproves Conjecture \ref{conj 1}. 

\subsection{Fourier dimension}
A lower bound of Beurling dimension was obtained in Theorem 1.3 \cite{IosLaiLiuWym2019} that if $\mu$ is a frame spectral measure with spectrum $\Lambda$, then
\begin{equation}
\dimB \Lambda\ge \dim_{F}\mu,
\end{equation}
where $\dim_F$ is the Fourier dimension which is defined by the formula
$$
\dim_{F}\mu:=\sup \left\{0\le s\le d: \exists C, \forall \xi,  |\widehat{\mu}(\xi)|\le C|\xi|^{-s/2} \right\}.
$$
Combining this with Lemma \ref{lem:restriction}, we obtain the following Theorem.
\begin{thm}\label{thm:lower bound}
	Let $\mu$ be a Borel measure on $\R^d$. Suppose that $\mu$ is a frame spectral measure with frame spectrum $\Lambda$. Then we have
	$$
	\dimB \Lambda\ge \sup_{\mu(K)>0, \mu(\partial K)=0}{\dim}_F~ \mu_K.
	$$
\end{thm}
A direct consequence of Theorems \ref{main thm 2} and \ref{thm:lower bound} is the following necessary condition for frame spectral measures.
\begin{cor}\label{cor:necessary condition}
	Let $\mu$ be a Borel measure on $\R^d$. Suppose that $\mu$ is a frame spectral measure with frame spectrum $\Lambda$. Then we have
	\begin{equation}\label{eq:necessary condition}
	\sup_{\mu(K)>0, \mu(\partial K)=0}{\dim}_F~ \mu_K\le \inf_{\mu(K)>0, \mu(\partial K)=0}\overline{\dimE}~ \mu_K.
	\end{equation}
\end{cor}
Using Corollary \ref{cor:necessary condition}, some non frame spectral measures could be shown as follows.
\begin{cor}\label{cor:sum}
	Let $\mu, \nu$ and $\rho$ be Borel measures on $\R^d$. Suppose that $\dim_F\mu>\overline{\dim}_e \nu$ and $\mu(\text{supp}(\nu+\rho))=\nu(\text{supp}(\mu+\rho))=0$. Then the measure $\mu+\nu+\rho$ is not a frame spectral measure.
\end{cor}

We end up this section by proposing some open questions.
We remark that if $\mu$ is absolutely continuous or discrete, then the equality holds in \eqref{eq:necessary condition} in Corollary \ref{cor:necessary condition}. Hence we might ask the same question for singular continuous measures:
\begin{que}
	Does there exist a frame spectral measure which is singular continuous and the equality holds in \eqref{eq:necessary condition}?
\end{que}  
As far as I know, we don't yet have an example of frame spectral measures that have non-zero Fourier dimension. Thus we might ask the following question.
\begin{que}
	Does there exist a frame spectral measure that is singular continuous and has non-zero Fourier dimension?
\end{que}


\section{Potential examples}\label{Sec:Potential examples}

In this section, we apply Theorems \ref{main thm 1} and \ref{main thm 2} for various measures. Since few concrete examples of frame spectral measures are known, the results in this section might be helpful to find new examples of measures of different type. 

\subsection{Self-affine measures}
A function $\varphi$ is called a \textit{contraction} on a complete metric space $X$ with metric $d$ if $d(\varphi(x),\varphi(y))<d(x,y)$ holds for every $x\not=y\in X$. If $\{\varphi_i\}_{1\le i\le N}$ are contractions of $X$ it is well-known that there exists a unique non-empty compact set $K\subset X$ such that $K=\cup_{1\le i\le N} \varphi_i(K)$ (see \cite{Hut1981}). In this circumstance the tuple $\{\varphi_i\}_{1\le i\le N}$ is called an \textit{iterated function system} (IFS) and $F$ its attractor. A central problem in the study of iterated function systems is to calculate or estimate the dimension of the attractor $F$ for various notions of fractal dimension, most especially the Hausdorff dimension. Particular interest has been given to the case of affine iterated function system, where the ambient space $X$ is given by $\R^d$ and the contractions $\varphi_i$ take the form $\varphi_i: x\mapsto A_ix+b_i$ for certain (usually invertible) linear maps $A_i\in \mathcal{L}(\R^d, \R^d)$ and vectors $b_i$. The associated attractors are called \textit{self-affine}.

Very recently, Hochman and Rapaport \cite{HocRap2019} proved that if $\mu$ is a self-affine measure in the plane whose defining IFS acts totally irreducibly and satisfies an exponential separation condition, then its dimension is equal to its Lyapunov dimension. Applying Theorem \ref{main thm 1}, we get the following corollary.

\begin{cor}
	Let $\Phi=\{\varphi_i \}_{i\in J}$ be a finite system of invertible affine contractions of $\R^2$. Suppose that $\Phi$ has no common fixed point, satisfies the non-conformality and total irreducibility assumptions, and is exponentially separated. Let $p$ be a positive probability vector. Let $\mu=\sum p_i\cdot \varphi_i \mu$ be the associated self-affine measure. Assume that $\mu$ is a frame spectral measure with frame spectrum $\Lambda$. Then 
	$$
	\dim \Lambda\le \min\{2, \dim_{L} \mu \},
	$$
	where $\dim_{L}$ stands for Lyapunov dimension.
\end{cor}

In general, Feng \cite{Feng2019} proved that every ergodic invariant measure for an affine IFS is exact dimensional, and its Hausdorff dimension satisfies a Ledrappier-Young type formula. Applying Theorem \ref{main thm 1}, we have the following corollary.

\begin{cor}
	Let $\mu$ be an ergodic invariant measure for an affine IFS. Suppose that $\mu$ is a frame spectral measure with frame spectrum $\Lambda$. Then we have 
	$$
	\dim \Lambda\le \dim \mu.
	$$
\end{cor}


\subsection{$\times \beta$-invariant measures}

Let $T_\beta$ be the multiplication by $\beta$ modulo one on the unite interval. It is well known that if $\beta$ is Pisot number and $\mu$ is $T_\beta$-invariant, then $\mu$ is exact dimensional. Thus we have the following corollary as a direct consequence of Theorem \ref{main thm 1}.

\begin{cor}
 Let $\beta$ be a Pisot number.	If $\mu$ is a $T_\beta$-invariant frame spectral measure with spectrum $\Lambda$, then we have
 $$
 \dim \Lambda\le \dim \mu.
 $$
\end{cor}


\subsection{Measures of ``mixed type" }

It is used to be conjectured that a pure type phenomenon should also exist within
the class of singular continuous measures, that is to say, all frame spectral measures are exact dimensional. The first counterexample was constructed by Lev \cite{Lev2018} as follows.

Let $\mu\in \mathcal{P}(\R^n)$ and $\nu\in \mathcal{P}(\R^m)$. We define a new measure $\rho$ on $\R^{n+m}$ by
$$
\rho=\mu\times \delta_0+\delta_0\times \nu,
$$
where $\delta_0$ denotes the Dirac measure at the origin. It is not hard to see that $\rho$ is the singular measure whose support is contained in $(\R^n\times \{0\})\cup (\{0\}\times \R^m)$. The frame spectral measure of ``mixed type" is constructed in the following theorem.

\begin{thm}[\cite{Lev2018}, Theorem 2.1]\label{thm:Nir lev}
	Assume that two measures $\mu$ and $\nu$ are continuous frame spectral measures. Then the measure $\rho$ defined above is also a frame spectral measure. 
\end{thm}

Applying Theorem \ref{main thm 1} to the measure $\rho$ of ``mixed type" in Theorem \ref{thm:Nir lev}, we have the following corollary.
\begin{cor}
	Let $\mu, \nu, \rho$ be defined in Theorem \ref{thm:Nir lev}. Suppose that $\Lambda$ is a frame spectrum of $\rho$. Then we have $$\dim~\Lambda\le \min\{\overline{\dim}_e \mu, \overline{\dim}_e \nu  \}.$$
\end{cor}


\section*{Acknowledgments} 
We would like to thank Chun-Kit Lai for many valuable remarks. We are also grateful for the anonymous referees for their valuable suggestions and comments. This work was partially supported by the Centre of Excellence in Analysis and Dynamics Research funded by the Academy of Finland.


\end{document}